\documentclass[12pt,reqno]{amsart}

\usepackage{amssymb,amsfonts,latexsym,amstext,epsfig,color}

\newtheorem{theorem}{Theorem}
\newtheorem{lemma}[theorem]{Lemma}

\begin{document}

\title[Burning spiders and path-forests]{Bounds on the burning numbers of spiders and path-forests}
\author{Anthony Bonato}
\address{Department of Mathematics\\
Ryerson University\\
Toronto, ON\\
Canada, M5B 2K3} \email{abonato@ryerson.ca}
\author{Thomas Lidbetter}
\address{Department of Management Science \& Information Systems\\
Rutgers University\\
Newark, NJ\\
USA, 07102} \email{tlidbetter@business.rutgers.edu}

\thanks{The first author gratefully acknowledge support from NSERC}

\begin{abstract}
Graph burning is one model for the spread of memes and contagion in social networks. The corresponding graph parameter is the burning number of a graph $G$, written $b(G)$, which measures the speed
of the social contagion. While it is conjectured that the burning number of a connected graph of order $n$ is at most $\lceil \sqrt{n} \rceil$, this remains open in general and in many graph
families. We prove the conjectured bound for spider graphs, which are trees with exactly one vertex of degree at least 3. To prove our result for spiders, we develop new bounds on the burning
number for path-forests, which in turn leads to a $\frac 3 2$-approximation algorithm for computing the burning number of path-forests.
\end{abstract}

\maketitle

\section{Introduction}

Internet memes spread quickly across social networks such as Facebook and Instagram. The burning number of a graph was introduced as a simple model of spreading memes or other kinds of social
contagion in~\cite{BJR,thez}. The smaller the burning number is, the faster a contagion (such as a meme, news, or gossip) spreads in the network.

Given a graph $G$, the burning process on $G$ is a discrete-time process defined as follows. Initially, at time $t=0$ all vertices are unburned. At each time step $t \geq 1$, one new unburned vertex
is chosen to burn (if such a vertex is available); such a vertex is called a \emph{source of fire}. If a vertex is burned, then it remains in that state until the end of the process. Once a vertex is
burned in round $t$, in round $t+1$ each of its unburned neighbors becomes burned. The process ends when all vertices of $G$ are burned (that is, let $T$ be the smallest positive integer such that
there is at least one vertex not burning in round $T-1$ and all vertices are burned in round $T$). The \emph{burning number} of a graph $G$, denoted by $b(G)$, is the minimum number of rounds needed
for the process to end. Note that with our notation, $b(G) = T$. The vertices that are chosen to be burned are referred to as a \emph{burning sequence}; a shortest such sequence is called
\emph{optimal}. Note that optimal burning sequences have length $b(G).$

For example, for the path $P_4$ with vertices $\{v_1, v_2, v_3, v_4\}$, the sequence $(v_2, v_4)$ is an optimal burning sequence; see Figure~\ref{p4}.
\begin{figure}[h]
\begin{center}
\epsfig{figure=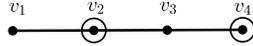,width=1.4in,height=.32in}
\caption{Burning the path $P_4$ (the open circles represent burned vertices).} \label{p4}
\end{center}
\end{figure}

It is evident that $b(G) \le \mathrm{rad}(G) + 1$, where $\mathrm{rad}(G)$ is the radius of $G.$ However, this bound can also be far from being tight; for example, for a path on $n$ vertices we have
that $\mathrm{rad}(P_n)=\lfloor \frac{n}{2} \rfloor$, whereas $b(P_n)=\lceil \sqrt{n} \rceil$. Results from \cite{BJR} give the bound $b(G) \leq 2\lceil{n}^{1/2}\rceil - 1$, where $G$ is connected of
order $n.$

An important conjecture in graph burning, first stated \cite{BJR}, is the following. \vspace{0.1in}

\noindent \textbf{Burning number conjecture}: For a connected graph $G$ of order $n$, $$b(G)\leq \lceil {n}^{1/2}\rceil.$$ \vspace{0.1in}

If the burning number conjecture holds, then paths are examples of graphs with highest burning number. Indeed, as shown in \cite{BJR}, the conjecture holds if it is satisfied by trees. So far, the
conjecture has resisted attempts at its resolution. In \cite{bessy2}, it was proved that $$b(G)\leq \sqrt{\frac{32}{19}\cdot \frac{n}{1-\epsilon}}+\sqrt{\frac{27}{19\epsilon}}$$ and
$$b(G)\leq \sqrt{\frac{12n}{7}}+3\approx 1.309 \sqrt{n}+3$$ for every connected graph $G$ of order $n$ and every $0<\epsilon<1$. These bounds were improved in \cite{LL} to $$b(G)\le \bigg \lceil
\frac{-3+\sqrt{24n+33}}{4} \bigg \rceil.$$ We also note that a randomized notion of burning was studied in \cite{rburn}, and burning was considered in circulant graphs in \cite{shannon}.

In this paper, we settle the burning number conjecture for the class of spider graphs, which are trees with exactly one vertex of degree strictly greater than two; see Theorem~\ref{maint}. While
spiders might initially appear to be an elementary graph class in the context of graph burning, it was shown in \cite{bjr3} that computing the burning number on spiders is \textbf{NP}-complete. The
main ingredient in our proof of the conjecture for spiders relies on new bounds on the burning number of path-forests (that is, disjoint unions of paths); see Lemmas~\ref{thm:path-forest}
and~\ref{lem:path-forest2}. The bounds provided in our results improve on the known bound given in \cite{BJR} of $b(G) \le \lceil n^{1/2} \rceil +t-1$, where $G$ is a path-forest of order $n$ with
$t$ components. As shown in \cite{bjr3}, the problem of computing the burning number of path-forests is also \textbf{NP}-complete. Our bounds provide a $\frac 3 2$-approximation algorithm for
computing the burning number of path-forests; see Theorem~\ref{thgreed}.

All graphs we consider are simple, finite, and undirected. We say that a collection of subsets of vertices of a graph $G$ {\em covers} $G$ if the union of those subsets is $V(G)$. For a vertex $v$
and a non-negative integer $r$, the {\em $k$th closed neighborhood} $N_r[v]$ of $v$ is defined as the set of all vertices within distance $r$ of $v$ (including $v$ itself). For background on graph
theory, see \cite{west}.

\section{A greedy algorithm for burning path-forests}

In this section, we derive two upper bounds on the burning number of a path-forest $G$, and show that these upper bounds define an algorithm for burning $G$. We then show that these bounds imply that
the algorithm has an approximation ratio of $\frac 3 2$. That is, the number of steps it requires to burn a path-forest $G$ is at most $\frac 3 2 b(G)$. Note that the algorithm is not optimal;
indeed, as referenced in the introduction, the problem of finding $b(G)$ if $G$ is a path-forest is \textbf{NP}-complete (see \cite{bjr3}).

We first point out that there is a simple {\em lower bound} on the burning number $b(G)$ of a path-forest of order $n$ with $t$ components:
\begin{align}
b(G) \ge \max \{\lceil n^{1/2} \rceil ,t \}. \label{eq:lb}
\end{align}
This inequality follows from the fact that the burning number of $G$ is at least the burning number, $\lceil n^{1/2} \rceil$ of a path of length $n$, and the number of fires required to burn any
graph with $t$ components is at least $t$.

Before stating and proving the lemmas we make a simple but useful lemma, which gives an equivalence between graph burning and a certain covering problem. This is analogous (but not identical to)
Theorem~2 and Corollary~3 from \cite{BJR}.

\begin{lemma} \label{obs:covering}
A graph $G$ satisfies $b(G) \le M$ for some integer $M$ if and only if for some $k \ge 1$,
\[
V(G) = N_{r_1}[v_1] \cup \ldots \cup N_{r_k}[v_k],
\]
for some vertices $v_i$ and integers $r_i \le M-i,i=1,\ldots,k$.
\end{lemma}

\begin{proof} For the forward implication, take an optimal burning sequence $(v_1, \ldots,v_{b(G)})$, set $k=b(G)$ and $r_i = b(G) -i$ for $1 \le i \le b(G)$. Then the union over $i=1,\ldots,k$ of the neighborhoods $N_{r_i}[v_i]$ is $V(G)$. For the
reverse implication, the sequence $(v_i)_{i=1}^M$ burns the graph in time at most $M$, where for $k+1 \le i \le M$, the vertex $v_i$ can be chosen to be any arbitrary vertex that has not yet appeared
in the sequence. \end{proof}

Lemma~\ref{obs:covering} implies that in order to prove a bound of the form $b(G) \le M$ using induction (assuming it has been verified for some base cases), it is sufficient to remove some
neighborhood $N_{r_1}[v_1]$ from $G$, where $r_1 \le M-1$, then show that the burning number of the new graph is at most $M-1$. We will use this proof technique for the following two lemmas.

\begin{lemma} \label{thm:path-forest}
If $G$ is a path-forest of order $n$ with $t \ge 1$ components, then
\begin{align}
b(G) &\le \left \lfloor \frac{n}{2t} \right \rfloor  + t. \label{eq:ub}
\end{align}
\end{lemma}

We observe that the bound~(\ref{eq:ub}) in the lemma is tight when $G$ consists of a set of $t$ disjoint paths of order 1 (that is, if $G$ is a co-clique).

\begin{proof}[Proof of Lemma~\ref{thm:path-forest}] The proof is by induction on $n$. For the base cases, we consider the collection of path-forests such that $\lfloor n/(2t) \rfloor = 0$. In this case, we must show that $b(G) \le t$ (which necessarily means that $b(G)=t$,
by~(\ref{eq:lb})). Let the orders of the $t$ components of $G$ be $a_1,\dots,a_t$, where the $a_i$ are arranged in non-decreasing order. We must show that the $j$th component has order at most
$2j-1$, so that it can be covered by a closed neighborhood of radius $j-1$.

Suppose not, then it must be that for some $j$, we have $a_j \ge 2j$. In that case, for $i>j$, we must have $a_i \ge 2j$ and because for all $i$ we have $a_i \ge 1$, it follows that
	\begin{align*}
	n &\ge (j-1)+(t-j+1)2j \\
	&= 2t + 2t(j-1) + 3j - 1 - 2j^2 \\
	& \ge 2t + j-1,
	\end{align*} where the last inequality holds since $t\ge j$.  This contradicts $\lfloor n/(2t) \rfloor =0$, so we must have that $b(G) \le t$.

Now suppose $ \lfloor n/(2t) \rfloor \ge 1$, and suppose the lemma is true for smaller values of $n$. Let $M= \lfloor n/(2t) \rfloor + t$. By Lemma~\ref{obs:covering} and the remark following it, it is sufficient to remove some closed neighborhood of radius at most $M-1$ from $G$, and show that the new graph has burning number at most $M-1$.


We consider two cases.

\medskip

\noindent {\em Case 1.} The radius of the largest components of $G$ is greater than $M-1$.

\medskip

In this case, let $G'$ be the graph with $t$ components obtained by removing a subpath of order $2M-1$ from a largest component of $G$. Then the order of $G'$ satisfies
\begin{align}
v(G') &\le n-(2M-1) \le n-2t, \label{eq:2}
\end{align}
with the second inequality following from $\lfloor n/(2t) \rfloor \ge 1$. Hence, by induction,
\begin{align*}
b(G') &\le \left \lfloor \frac{v(G')}{2t} \right \rfloor + t\\
& \le \left \lfloor \frac{n-2t}{2t} \right \rfloor +t \\
& = M-1,
\end{align*}
where the second inequality holds by~(\ref{eq:2}).

\medskip

\noindent {\em Case 2.} The radius of the largest components of $G$ is at most $M-1$.

\medskip

In this case, let $G''$ be the graph with $t-1$ components obtained by removing a largest component of $G$. We must prove that $b(G'') \le M -1$. Since the average order of the components of $G$ is
$n/t$, the largest components of $G$ must have order at least $\lceil n/t \rceil$. Hence, the order $v(G'')$ of $G''$ satisfies
\begin{align}
v(G'') & \le n- \lceil n/t \rceil \le n(t-1)/t. \label{eq:1}
\end{align}
Thus, by induction we have that
\begin{align*}
b(G') &\le \left \lfloor \frac {v(G'')}{2(t-1)} \right \rfloor +t-1 \\
& \le M-1,
\end{align*}
where the second inequality follows from~(\ref{eq:1}).
\end{proof}


For small $t$ (roughly, $t \le n^{1/2}/2$), the bound~(\ref{eq:ub}) is worse than the simpler bound $b(G) \le \lceil n^{1/2} \rceil + t -1$. We give an improved bound on the burning number for small
$t$ in the next lemma.

\begin{lemma} \label{lem:path-forest2}
	If $G$ is a path-forest of order $n$ with $t \le \lceil n^{1/2} \rceil$ components, then
	\begin{align}
	b(G) \le \left \lceil n^{1/2} + \frac {t-1} 2 \right \rceil. \label{eq:ub2}
	\end{align}
\end{lemma}

\begin{proof}
We proceed by induction on $n$, and first verify the lemma for $t = \lceil n^{1/2} \rceil$ and $t= \lfloor n^{1/2} \rfloor$ (which includes the case $n=1$). For $t = \lceil n^{1/2} \rceil$, applying
Lemma~\ref{thm:path-forest},
\begin{align*}
b(G) &\le \left \lfloor \frac{n}{2t} \right \rfloor + t \\
& =  \left \lfloor \frac{n}{2 \lceil n^{1/2} \rceil}  \right  \rfloor + \lceil n^{1/2} \rceil .
\end{align*}
It is an elementary exercise to verify that this bound is no greater than $\left \lceil n^{1/2} + (\lceil n^{1/2} \rceil-1)/2 \right \rceil$. Similarly, we can apply Lemma~\ref{thm:path-forest} to
the case $t= \lfloor n^{1/2} \rfloor$.

Now suppose that $t < \lfloor n^{1/2} \rfloor$ so that $t \le \lfloor n^{1/2} \rfloor - 1$ and assume the theorem is true for smaller values of $n$.  Analogously to Lemma~\ref{thm:path-forest}, we
remove from $G$ a closed neighborhood of radius at most $M-1$, where $M= \lceil n^{1/2} + (t-1)/2\rceil$, and show that the burning number of the new graph is at most $M-1$.

We consider two cases.

\medskip

\noindent {\em Case 1.} The largest components of $G$ have radius greater than $\lceil n^{1/2} \rceil -1$.

\medskip

In this case, we remove a path of radius $\lceil n^{1/2} \rceil -1  \le M-1$ from $G$ to give a new path-forest $G'$ with $t$ components and order $v(G') = n- 2\lceil n^{1/2} \rceil + 1 \le (n^{1/2} - 1)^2$.
Since $t \le \lfloor n^{1/2} \rfloor - 1 \le v(G')^{1/2}$, by induction we have
\[
b(G') \le \left \lceil (n^{1/2}  -1 )+ \frac{t-1} 2 \right \rceil = M-1.
\]

\medskip

\noindent {\em Case 2.} The largest components of $G$ have radius at most $\lceil n^{1/2} \rceil -1$.

\medskip

In this case, we remove a largest component of $G$ to leave a path-forest $G''$ with $t-1$ components. The largest components of $G$ must have order at least the ceiling of the average order of the
components of $G$, which is $\lceil n/t \rceil  \ge \lceil n^{1/2} \rceil$ (the inequality following from $t \le \lfloor n^{1/2} \rfloor - 1$). This implies that the order of $G''$ is at most $v(G'')
\le n-\lceil n^{1/2} \rceil \le ( n^{1/2}  -1/2)^2$.

Further, since $v(G'') \ge n-2 \lceil n^{1/2} \rceil +1$, we must have that $t-1 \le v(G'')^{1/2}$. Hence, by induction,
\[
b(G'') \le \left \lceil (n^{1/2} -1/2 )+ \frac{(t-1)-1}{2} \right \rceil =M-1,
\]
and the proof follows.
\end{proof}

Lemmas~\ref{thm:path-forest} and~\ref{lem:path-forest2} and their proofs define the algorithm \textsf{GREEDY} for a path-forest $G$ of order $n$ with $t$ components. The algorithm recursively covers
$G$ with closed neighborhoods, as in the proofs of Lemmas~\ref{thm:path-forest} and~\ref{lem:path-forest2}, and these closed neighborhoods define a burning sequence, as described in the proof of the
Lemma~\ref{obs:covering}.

More precisely, \textsf{GREEDY} can be described as follows.

\begin{enumerate}
	\item If $t \ge \lfloor n^{1/2} \rfloor$, then set $r = \lfloor n/(2t) \rfloor +t-1$; else, set $r= \lceil n^{1/2} \rceil -1$.
	\item If the largest component of $G$ has radius at most $r$, then remove it to leave a new graph $G'$ with $t-1$ components; otherwise, remove a closed neighborhood of radius $r$ from a
largest component of $G$ to leave a new graph $G'$ with $t$ components.
	\item If $G'$ is empty, then stop; else, set $G=G'$, update $n$ and $t$ and return to step (1).
\end{enumerate}

We illustrate the algorithm with an example. We denote a path-forest $G$ by a sequence of positive integers, where coordinates specify the order of the components. We start with the
path-forest $(13,11,11),$ so that $t=3$ and $n=35$. Since $t < \lfloor n^{1/2} \rfloor$, we set $r = \lceil n^{1/2} \rceil - 1 =5$ and remove a closed neighborhood of radius $5$ from the component
of order $13$. This corresponds to the first vertex $v_1$ in the burning sequence, as depicted on the left side of Figure~\ref{example}. The new graph is $(2,11,11)$, and we still have $t < \lfloor
n^{1/2} \rfloor$, so we set $r = \lceil n^{1/2} \rceil -1= 4$ and remove a closed neighborhood of radius $4$ from one of the components of size $11$, leaving a new graph $(2,2,11)$. Again, the second
vertex $v_2$ in the burning sequence is depicted on the left of Figure~\ref{example}. Now we have  $t \ge \lfloor n^{1/2} \rfloor$, so we set $r =\lfloor n/(2t) \rfloor +t -1= 4$ and remove a closed
neighborhood of radius $4$ from the component of size $11$, to leave the graph $(2,2,2)$. We omit all the details, but the algorithm continues with $r=3$ in the next step, producing a new graph
$(2,2)$, then $r=2$, producing the graph $(2)$ and finally $r=1$ producing the empty graph. Each closed neighborhood that is removed defines a new vertex in the burning sequence, and these $6$
vertices $v_1\ldots,v_6$ are depicted on the left of Figure~\ref{example}. The vertices of the original graph $G$ can be written $$V(G)=N_5[v_1] \cup N_4[v_2] \cup N_4[v_3] \cup N_3[v_4] \cup
N_2[v_5] \cup N_1[v_6].$$ By Lemma~\ref{obs:covering}, we have that $b(G) \le 7$. The seventh vertex $v_7$ in the burning sequence is set to be some other arbitrary vertex.

\begin{figure}[h!]
	\begin{center}
		\epsfig{figure=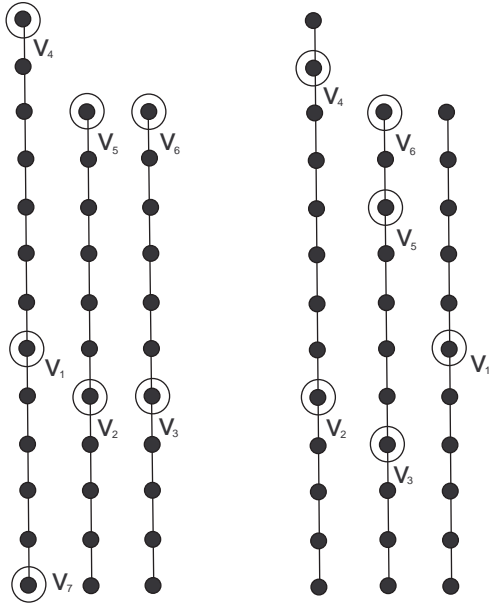}
		\caption{Two copies of the path-forest $(13,11,11)$ with a burning sequence from \textsf{GREEDY} on the left, and an optimal burning sequence on the right.} \label{example}
	\end{center}
\end{figure}

The algorithm is not necessarily optimal (as we would expect); indeed in this example, the burning number is equal to 6, and an optimal burning sequence is indicated on the right side of
Figure~\ref{example}.

Using the upper bounds of Lemmas~\ref{thm:path-forest} and~\ref{lem:path-forest2} and the lower bound~(\ref{eq:lb}), we can measure the performance of the algorithm \textsf{GREEDY}. These bounds are
plotted against $t$ in Figure~\ref{fig:plot} for $n=10,000$, where the lower bound is shown in red and the upper bounds are shown in orange and green. It is evident that the ratio between the upper
and lower bounds is greatest at $t=100=n^{1/2}$, when it is equal to approximately $3/2$.

\begin{figure}[h!]
	\begin{center}
		\epsfig{figure=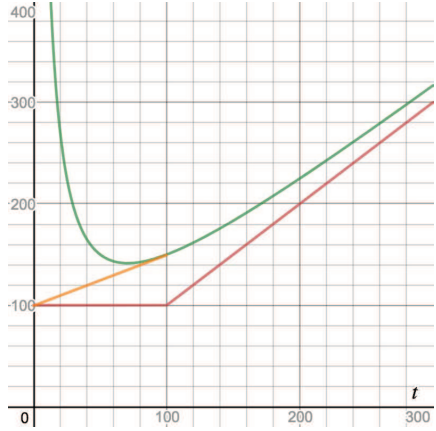,scale=0.2}
		\caption{Upper and lower bounds for the burning number of a path-forest.} \label{fig:plot}
	\end{center}
\end{figure}

We now show that the burning sequence produced from the algorithm \textsf{GREEDY} is a $\frac 3 2$-approximation for the optimal burning sequence of a path-forest.

\begin{theorem}\label{thgreed}
The number of steps $T_G$ required by \textsf{GREEDY} to burn a path-forest $G$ satisfies
\[
T_G \le \frac 3 2 b(G).
\]
\end{theorem}
\begin{proof}
Suppose first that $t \le \lceil n^{1/2} \rceil$. Then we have that
\begin{align*}
\frac{T_G}{b(G)} & \le \frac{ \lceil n^{1/2} + (t-1)/2 \rceil}{\lceil n^{1/2} \rceil} \\	
& \le \frac{ \lceil n^{1/2} + (\lceil n^{1/2} \rceil -1)/2 \rceil}{\lceil n^{1/2} \rceil} \\
& \le \frac 3 2,	
\end{align*}
where the first inequality holds by Lemma~\ref{lem:path-forest2} and~(\ref{eq:lb}), and the second inequality holds since $t \le \lceil n^{1/2} \rceil$.

In the case that $t \ge \lceil n^{1/2} \rceil +1 $, we have that
\begin{align*}
\frac{T_G}{b(G)} & \le \frac{ \lfloor n/(2t) \rfloor + t}{t} \\
& \le \frac{ \left \lfloor \frac{n}{2(\lceil n^{1/2}\rceil +1)} \right \rfloor}{\lceil n^{1/2}\rceil +1} + 1  \\
& \le \frac 3 2.
\end{align*}
where the first inequality holds by Lemma~\ref{thm:path-forest} and~(\ref{eq:lb}), and the second inequality holds since $t \ge \lceil n^{1/2} \rceil +1$.

\end{proof}

\section{Bounding the burning number of a spider}

Before turning to spiders, we prove a crucial ``meta-lemma'' which may be applied to families of graphs other than spiders.

\begin{lemma} \label{lem:meta}
	Let $\mathcal{G}$ be a family of connected graphs, and let $\hat{\mathcal{G}}$ be some subset of graphs in $\mathcal{G}$ with the following properties. For every $G$ in $\hat{\mathcal{G}}$, there is some vertex $v \in V(G)$ and some $r \le \lceil |V(G)|^{1/2} \rceil -1 $ such that either $N_r[v]=V(G)$ or

\begin{enumerate}
\item the closed neighborhood $N_r[v]$ has order at least $2 \lceil{|V(G)|^{1/2}} \rceil  -1 $ and
\item the graph induced by $V(G) \setminus N_r[v]$ is connected and contained in $\mathcal{G}$.
\end{enumerate}
If $b(G) \le \lceil{|V(G)|^{1/2}} \rceil  $ for all $G$ in $\mathcal{G} \setminus \hat{\mathcal{G}}$, then $b(G) \le  \lceil{|V(G)|^{1/2}} \rceil $ for all $G$ in $\mathcal{G}$.
	\end{lemma}

\begin{proof}
Let $G$ be a graph in $\mathcal{G}$. The proof is by induction on the order $n$ of $G$. To verify the lemma for the base cases, let $G$ be a member of $\mathcal{G}$ of minimum order. If $G$ is not contained in $\hat{\mathcal{G}}$, then certainly $b(G) \le \lceil{|V(G)|^{1/2}} \rceil$. Otherwise, $G$ is contained in $\hat{\mathcal{G}}$ and by the minimality of the order of $G$, there must be some vertex $v \in V(G)$ and some $r \le \lceil |V(G)|^{1/2} \rceil -1 $ such that $N_r[v]=V(G)$, in which case clearly $b(G) \le \lceil{|V(G)|^{1/2}} \rceil$.
	
Suppose $G$ is in $\hat{\mathcal{G}}$ and satisfies items (1) and (2). Let $G'$ be the subgraph induced by $V(G) \setminus N_r[v]$. By Lemma~\ref{obs:covering}, it is sufficient to show that $b(G')
\le \lceil{n^{1/2}} \rceil -1$. Note that $G'$ has order at most $(\lceil{n^{1/2}} \rceil -1)^2$, so if $G'$ is in $\hat{\mathcal{G}}$, then $b(G') \le \lceil{n^{1/2}} \rceil - 1$, by induction.
Alternatively, if $G'$ is in $\mathcal{G} \setminus \hat{\mathcal{G}}$, then $b(G') \le \lceil{n^{1/2}} \rceil - 1$ by the premises of the lemma.
\end{proof}

Observe that Lemma~\ref{lem:meta} can be used to prove that the burning number of a path of order $n$ is at most $\lceil n^{1/2} \rceil$, by taking $\mathcal{G}=\hat{\mathcal{G}}$ to be the family of
all paths.

We can now prove that the burning number of a spider is at most $\lceil n^{1/2} \rceil$. Recall that the unique vertex of the spider with degree at least 3 is called the {\em head} and the paths from
the head to the leaf nodes are called the {\em arms} (which do not contain the head). We call the {\em length} of an arm the distance along that arm from the head to the leaf. For ease of notation,
we may identify a spider with a positive integer sequence $(a_1,a_2,\ldots, a_m)$, where for all $1\le i \le m$, $a_i$ is the length of each arm.

We will prove the result first for a set of smaller cases.

\begin{lemma} \label{lemsmall}
If $G$ is a spider of order $n \le 25$, then $b(G) \le \lceil n^{1/2} \rceil$.
\end{lemma}

\begin{proof}
Without loss of generality, we may assume $n$ is one of the square integers 1, 4, 9, 16, or 25. The cases $n=1,4$ are trivial and so omitted.

Suppose first that $n=9.$ By Lemma~\ref{lem:meta}, we may assume that the length of the arms is at most four. If the first source of fire is the head of the spider, then we will burn any arm with
length at most 2. Thus, we may assume without loss of generality each arm has length at least 3. But then we would have at least 3 arms of length at least 3, which would force the spider to have 10
or more vertices. Hence, the case $n=9$ follows.

For the case $n=16,$ by an analogous discussion, each arm has length at most 6 and at least 4, and there are at most 3 arms. It is straightforward to check the only non-trivial cases are $(6,5,4)$
and $(5,5,5).$ For each spider, the burning number is 4; see Figure~\ref{fig16}.
\begin{figure}[h!]
\begin{center}
\epsfig{figure=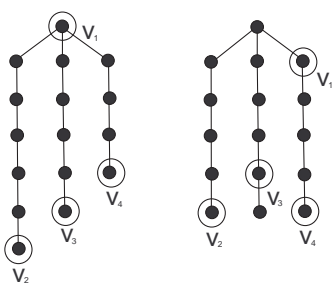,scale=0.8}
\caption{The spiders $(6,5,4)$ and $(5,5,5)$ with optimal burning sequences $(v_1,v_2,v_3,v_4)$.} \label{fig16}
\end{center}
\end{figure}

\begin{figure}[h!]
\begin{center}
\epsfig{figure=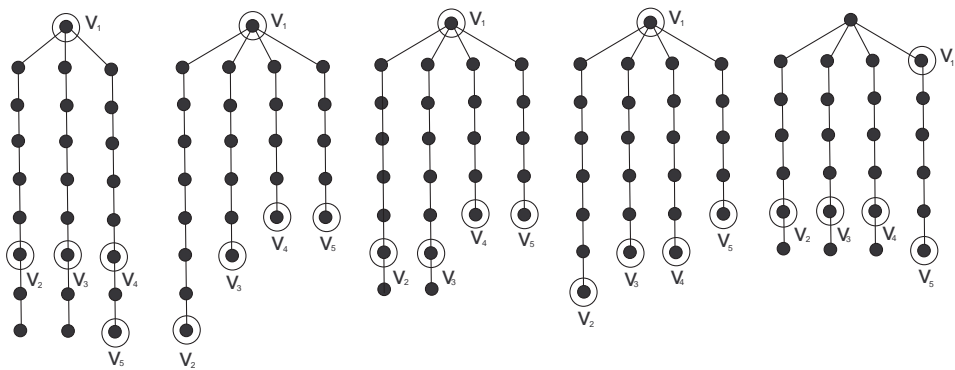,scale=0.8}
\caption{The spiders $(8,8,8)$, $(8,6,5,5)$, $(7,7,5,5)$, $(7,6,6,5)$, and $(6,6,6,6)$ with optimal burning sequences $(v_1,v_2,v_3,v_4,v_5)$.} \label{fig25}
\end{center}
\end{figure}
For the case $n=25,$ by an analogous discussion, each arm has length at most 8 and at least 5, and there are at most 4 arms. The only non-trivial cases here are $(8,8,8)$, $(8,6,5,5)$, $(7,7,5,5)$,
$(7,6,6,5)$, and $(6,6,6,6).$ Each of these trees has burning number 5; see Figure~\ref{fig25}.
\end{proof}

We now come to the main result of the paper.

\begin{theorem}\label{maint}
	The burning number of a spider graph $G$ of order $n$ satisfies $b(G) \le \lceil n^{1/2} \rceil$.
\end{theorem}

\begin{proof} Set $\alpha = \lceil n^{1/2} \rceil$, and we may assume $\alpha > 5$ by Lemma~\ref{lemsmall}.	Applying Lemma~\ref{obs:covering}, we wish to express the vertices $V(G)$ of the spider as $V(G) = N_{r_1}[v_1] \cup \ldots \cup
N_{r_k}[v_k]$, where $r_i \le \alpha-i$ for $i=1,\ldots,k$. We only need to prove the theorem in the case that all the arms of the spider have length at most $2 \alpha - 2$, because the alternative
case then follows automatically from Lemma~\ref{lem:meta}.

There are two cases.

\medskip

\noindent \textit{Case 1.} The spider $G$ has $\alpha-1$ arms of length $\alpha+1$.

\medskip

In this case, let $v$ be any vertex adjacent to the head of the spider, and consider the graph $G'$ obtained by removing $N_{\alpha-1}[v]$. The graph $G'$ is a path-forest with $\alpha-2$ paths of
order $2$ and one path of order $1$, so $\lfloor v(G')/(2t) \rfloor = 0$, where $t=\alpha -1$. By Lemma~\ref{thm:path-forest}, $b(G') \le \alpha-1 $, so $b(G) \le \alpha$.
	

	
	
\medskip

\noindent \textit{Case 2.} The spider $G$ does not have $\alpha-1$ arms of length $\alpha+1$.		

\medskip

Consider the graph $H$ obtained by removing $N_{\alpha-1}[v]$, where $v$ is now the head of the spider. The graph $H$ is a path-forest with $t$ components, for some $t \ge 1$ and since the
arms of $G$ have length at most $2\alpha -2$, each of the components $H_1,\ldots,H_t$ of $H$ have order at most $\alpha -1 $. Arrange the $H_k$ in decreasing order, and let $v_k$ be a center of
$H_k$. If $t \le \alpha /2$, then each $H_k$ can be covered by the neighborhood $N_{\alpha-k}[v_k]$, since $2(\alpha - k) -1 \ge 2(\alpha -t) -1 \ge \alpha -1$, and we are done. In fact, we can do
better than this: if $t=(\alpha+1)/2$ (so that $\alpha$ must be odd), then for $1 \le k \le t-1$, each $H_k$ can be covered by the neighborhood $N_{\alpha-k}[v_k]$ and $H_t$ can be covered by the
union of $N_{(\alpha - 1)/2}[v_t]$ and $N_1[x]$, where $x$ is the furthest leaf node from $v_k$ in $H_k$. (Note that the radii, 1 and $(\alpha - 1)/2$ of these neighborhoods are distinct, since
$\alpha \ge 5$.)

Hence, for the remainder of the proof we may assume that $t \ge \lfloor \alpha/2 + 3/2 \rfloor$. We will use Lemma~\ref{thm:path-forest} to show that $b(H) \le \alpha -1$, which is sufficient to
prove the theorem. Note that since we have removed at least $t (\alpha -1)+1$ vertices from $G$, the order $v(H)$ of $H$ satisfies
	\begin{align}
	v(H) &\le n-t (\alpha -1)-1 \nonumber\\
	&\le (\alpha - 1)(\alpha + 1 - t). \label{eq:v(H)-ub}
	\end{align}
Hence, we must have $t \le \alpha$. But since $v(H) \ge t$, it is not possible that $t = \alpha$, because then, by~(\ref{eq:v(H)-ub}), $v(H) \le \alpha -1 < t$.
	
	Suppose $t = \alpha - 1$. By~(\ref{eq:v(H)-ub}), we must have $v(H) \le 2(\alpha -1)$, so $v(H)/t = 2$. If all the components of $H$ have order $2$, we are in {\em Case 1}, so there must be a component of order 1. In this case, it is straightforward to show that $b(H) \le t = \alpha -1 $, so that $b(G) \le \alpha$.
	
	Hence, we may assume that $\lfloor \alpha/2 +3/2 \rfloor \le t \le \alpha -2$. By Lemma~\ref{thm:path-forest} and (\ref{eq:v(H)-ub}), we have that
	\begin{align}
	b(H) &\le \left \lfloor \frac{v(H)}{2t} \right \rfloor + t \nonumber\\
	&\le \left \lfloor \frac{(\alpha - 1)(\alpha + 1 - t)}{2t} \right \rfloor + t. \label{eq:large-t}
	\end{align}
	By considering the right side of~(\ref{eq:large-t}) as a function of real non-negative $t$, it may be shown using elementary calculus that it
 has a minimum for some $t$ that lies between $\lfloor \alpha/2 +3/2 \rfloor$ and $\alpha - 2$, and in this range it is maximized when $t$ is equal to these two end points.
	
	Substituting $t= \alpha - 2$ into~(\ref{eq:large-t}), we derive that
	\begin{align*}
	b(H) & \le \left \lfloor \frac{3(\alpha - 1)}{2(\alpha -2 )} \right \rfloor + \alpha - 2\\
	&= \left \lfloor \frac{\alpha +1}{2(\alpha -2 )} \right \rfloor + \alpha - 1\\
	& \le \alpha - 1,
	\end{align*}
	for $\alpha > 5$.
	
	Substituting $t=\lfloor \alpha/2 +3/2 \rfloor$ into~(\ref{eq:large-t}), we find that

\begin{eqnarray*}
	b(H) &\le & \left \lfloor \frac{(\alpha - 1)(\alpha + 1 - \lfloor \alpha/2 +3/2 \rfloor)}{2\lfloor \alpha/2 +3/2 \rfloor} \right \rfloor + \lfloor \alpha/2 +3/2 \rfloor \\
 & = & f(\alpha).
\end{eqnarray*}
An (omitted) analysis of cases for when $\alpha$ is even and odd shows that $f(\alpha) = \alpha - 1$ for all integers $\alpha
>5.$ Hence, this case follows and the proof is complete. \end{proof}

\end{document}